\documentclass[11pt]{article}

\usepackage[margin=1in]{geometry}         

\usepackage{longtable}
\usepackage{graphicx}
\usepackage{caption,subcaption}
\usepackage{stmaryrd}
\usepackage{framed}
\usepackage{amssymb}
\usepackage{epstopdf}
\usepackage{amsmath,amsfonts,amssymb}
\usepackage{enumerate}
\usepackage{multirow}
\usepackage{fixltx2e}
\usepackage{bbm}
\usepackage{url}
\usepackage{cite}
\usepackage{fullpage}
\usepackage{fancyhdr}
\usepackage{empheq}
\usepackage{color}
\usepackage{float}
\usepackage{soul}
\usepackage{graphicx}
\usepackage[doc]{optional}
\usepackage{enumitem}
\usepackage{xcolor}
\usepackage[toc,page]{appendix}
\usepackage{theorem}
\usepackage{array}
\usepackage{booktabs}
\usepackage{epsfig}
\usepackage{latexsym}

\definecolor{labelkey}{rgb}{0,0.08,0.45}
\definecolor{refkey}{rgb}{0,0.6,0.0}

\definecolor{myblue}{rgb}{.9, .9, 1}

\newtheorem{theorem}{Theorem}[section]
\newtheorem{lemma}[theorem]{Lemma}
\newtheorem{corollary}[theorem]{Corollary}
\newtheorem{proposition}[theorem]{Proposition}
\newtheorem{definition}[theorem]{Definition}

\theoremstyle{plain}{\theorembodyfont{\rmfamily}

\theoremstyle{plain}{\theorembodyfont{\rmfamily}
}
\theoremstyle{plain}{\theorembodyfont{\rmfamily}
}
\theoremstyle{plain}{\theorembodyfont{\rmfamily}
}
\theoremstyle{plain}{\theorembodyfont{\rmfamily}
}

\theoremstyle{plain}{\theorembodyfont{\rmfamily}
}

\theoremstyle{plain}{\theorembodyfont{\rmfamily}
\newtheorem{customC}{Conceptual Algorithm}
\newenvironment{Calg}[1]
  {\renewcommand\thecustomC{#1}\customC}
  {\endcustomC}

\theoremstyle{plain}{\theorembodyfont{\rmfamily}
\newtheorem{customV}{Variant}
\newenvironment{variant}[1]
  {\renewcommand\thecustomV{#1}\customV}
  {\endcustomV}

\theoremstyle{plain}{\theorembodyfont{\rmfamily}

\def\RR{{\mathbb{R}}}

\def\NN{{\mathbb{N}}}

\def\mcF{{\mathcal{F}}}

\newcommand{\la}{\langle}
\newcommand{\ra}{\rangle}

\newcommand{\nexto}{\kern -0.54em}

\newcommand{\dZ}{{\cal Z \kern -0.7em Z}}
\newcommand{\dC}{{\rm\hbox{C \kern-0.8em\raise0.2ex\hbox{\vrule height5.4pt width0.7pt}}}}
\newcommand{\dQ}{{\rm\hbox{Q \kern-0.85em\raise0.25ex\hbox{\vrule height5.4pt width0.7pt}}}}

\newenvironment{retraitsimple}{\begin{list}{--~}{
 \topsep=0.3ex \itemsep=0.3ex \labelsep=0em \parsep=0em
 \listparindent=1em \itemindent=0em
 \settowidth{\labelwidth}{--~} \leftmargin=\labelwidth
}}{\end{list}}

\begin{document}
\title{A variant of forward-backward splitting method for the systems of inclusion problems}
\author{R. D\'iaz Mill\'an\footnote{ Federal Institute of Education, Science and
Technology, Goi\^ania, Brazil, e-mail: rdiazmillan@gmail.com}} \maketitle
\begin{abstract}
In this paper, we propose variants of forward-backward splitting method for solving the system of splitting  inclusion problem. We propose a conceptual algorithm containing three variants, each having a different projection steps. The algorithm consists in two parts, the first and main contains an explicit Armijo-type search in the spirit of the extragradient-like methods for variational inequalities. In the iterative process the operator forward-backward is computed only one time for each inclusion problem, this represent a great computational saving if we compare with Tseng's algorithm, because the computational cost of this operator is very high.   The second part of the scheme
consists in special projection steps. The convergence analysis of the proposed
scheme is given assu\-ming monotonicity on both operators, without assuming Lipschitz
continuity on the forward operators.

\bigskip


\end{abstract}
\section{Introduction}
First, we introduce a notation and some definitions. The inner product in
$\RR^n$ is denoted by $\la \cdot , \cdot \ra$ and the norm induced
by the inner product by $\|\cdot\|$. We denote by $2^{C}$ the power set of $C$. For $X$ a nonempty, convex and
closed subset of $\RR^n$, we define the orthogonal projection of $x$
onto $X$ by $P_X(x)$, as the unique point in $X$, such that $\|
P_X(x)-y\| \le \|x-y\|$ for all $y\in X$. Let $N_X(x)$
be the normal cone to $X$ at $x\in X$, i.e.,
$N_X(x)=\{d\in \RR^n\, : \, \la d,x-y\ra\ge 0 \;\; \forall
y\in X\}$. Recall that an operator $T:\RR^n \rightarrow 2^{\RR^n}$
is monotone if, for all $(x,u),(y,v)\in Gr(T):=\{(x,u)\in
\RR^n\times \RR^n : u\in T(x)\}$, we have $\la x-y, u-v \ra \ge 0,$
and it is maximal if $T$ has no proper monotone extension in the
graph inclusion sense.

In this paper, we propose a modified algorithm  for solving a system of splitting inclusion problem, for the sum of two operators. Given a finite family of pair of operators $\{ A_i,B_i \}_{i \in \mathbb{I}}$, with $\mathbb{I}=:(1,2,\cdots, m)$ and $m\in \NN$.
The system of inclusion problem consists in:
\begin{equation}\label{problema}
\mbox{Find} \ \ x\in \RR^n  \ \  \mbox{such that} \ \  0\in(A_i+B_i)(x) \ \ \mbox{for all} \ \ i\in \mathbb{I},
\end{equation}
where the operators $A_i:\RR^n \rightarrow \RR^n$ are point-to-point and monotone and the operators $B_i:\RR^n\rightarrow 2^{\RR^n}$ are point-to-set maximal monotone operators. The solution of the problem is given by the interception of the solution of each component of the system, i.e., $S_*=\cap_{i\in \mathbb{I}} S^i_*$, where $S^i_*$ is defined as $S^i_*:=\{x\in \RR^n: 0\in A_i(x)+B_i(x)\}$.

The problem \eqref{problema} is a generalization of the system of variational inequalities, introduced by I.V. Konnov in \cite{konnov1}, taking the operators $B_i = N_{C_i}$ for all $i\in \mathbb{I}$, which have been studied in \cite{konnov, gibali, gibali2, gibali3}. A generalization of this results have been studied in \cite{semenov,eslam}, where the hypothesis that all $A_i$ are Lipschitz continuous for all $i\in \mathbb{I}$, is assumed for the convergence analysis. In this paper we improve this result assuming only monotonicity  for all operators $A_i$, and maximal monotonicity for the operators $B_i$. Also, we improve the linesearch proposed by Tseng in \cite{tseng}, calculating in each tentative of find the step size, the operator forward-backward only one time for each inclusion problem of the system. This improves the algorithm in the computational sense, because this operator is very expensive to compute. The idea for this manuscript was motivated from the works \cite{rei-yun, phdthesis}.

Problem \eqref{problema} have many applications in operations research, mathematical physics, optimization and differential equations. This kind of problem have been deeply studied and has recently received a lot attention, due to the fact that many nonlinear problems, arising within applied
areas, are mathematically modeled as nonlinear operator system of equations and/or inclusions, which each one are decomposed as sum of two operators.
\section{Preliminaries}
In this section, we present some definitions and results needed for the
convergence analysis of the proposed algorithm. First, we state two well-known facts
on orthogonal projections.
\begin{proposition}\label{proj}
Let $X$ be any nonempty, closed and convex set in $\RR^n$, and $P_X$ the orthogonal projection onto $X$. For all $x,y\in \RR^n$ and all $z\in X $ the following hold:
\begin{enumerate}
\item $ \|P_{X}(x)-P_{X}(y)\|^2 \leq \|x-y\|^2-\|(P_{X}(x)-x)-\big(P_{X}(y)-y\big)\|^2.$
\item $\la x-P_X(x),z-P_X(x)\ra \leq 0.$
\item $P_X=(I+N_X)^{-1}.$
\end{enumerate}
\end{proposition}

\begin{proof}
 (i) and (ii) see    Lemma    $1.1$    and    $1.2$    in    \cite{zarantonelo}. (iii) See Proposition $2.3$ in \cite{bauch}.
\end{proof}

 In the following we state some useful results on maximal monotone operators.
\begin{lemma}\label{bound}
Let $T:dom(T)\subseteq \RR^n\rightarrow 2^{\RR^n}$ be a maximal monotone operator.
Then,
\begin{enumerate}
\item $Gr(T)$ is closed.
\item $T$ is bounded on bounded subsets of the interior of
its domain.
\end{enumerate}
\end{lemma}

\begin{proof}
\begin{enumerate}
\item See Proposition $4.2.1$(ii) in \cite{iusem-regina}.
\item See Lemma 5(iii) in \cite{yu-iu}.
\end{enumerate}
\end{proof}

\begin{proposition}\label{inversa}
Let $T:dom(T)\subseteq\RR^n \rightarrow 2^{\RR^n}$ be a point-to-set and maximal monotone operator. Given $\beta >0$ then the operator $(I+\beta\, T)^{-1}: \RR^n \rightarrow dom(T)$ is single valued and maximal monotone.
\end{proposition}
\begin{proof}
See Theorem $4$ in \cite{minty}.
\end{proof}
\begin{proposition}\label{parada}
Given $\beta>0$ and $A: dom(A)\subseteq \RR^n\to \RR^n$ be a monotone operator and $B: dom(B)\subseteq \RR^n\rightarrow 2^{\RR^n}$ be a maximal monotone operator, then
 $$x=(I+\beta B)^{-1}(I-\beta A)(x),$$ if and only if, $0\in (A+B)(x)$.
\end{proposition}
\begin{proof}
See Proposition $3.13$ in \cite{PhD-E}.
\end{proof}

\noindent Now we define the so called Fej\'er convergence.
\begin{definition}
Let $S$ be a nonempty subset of $\RR^n$. A sequence $\{x^k\}_{k\in \NN}\subset \RR^n$ is said to be Fej\'er convergent to $S$, if and only if, for all $x\in S$ there exists $k_0\ge 0$, such that $\|x^{k+1}-x\| \le \|x^k - x\|$ for all $k\ge k_0$.
\end{definition}

This definition was introduced in \cite{browder} and have been
further elaborated in \cite{IST} and \cite{borw-baus}. A useful result on Fej\'er sequences is the following.

\begin{proposition}\label{punto}
If $\{x^k\}_{k\in \NN}$ is Fej\'er convergent to $S$, then:
\begin{enumerate}
\item the sequence $\{x^k\}_{k\in \NN}$ is bounded;
\item the sequence $\{\|x^k-x\|\}_{k\in \NN}$ is convergent for all $x\in S;$
\item if a cluster point $x^*$ belongs to $S$, then the sequence $\{x^k\}_{k\in \NN}$ converges to $x^*$.
\end{enumerate}
\end{proposition}

\begin{proof}
(i) and (ii) See Proposition $5.4$ in \cite{librobauch}. (iii) See Theorem $5.5$ in \cite{librobauch}.
\end{proof}

\section{The Algorithm}\label{section3}
Let $A_i:\RR^n \rightarrow \RR^n$ and $B_i:\RR^n\rightarrow 2^{\RR^n} $ be maximal monotone operators, with $A_i$ point-to-point and $B_i$ point-to-set, for all $i\in \mathbb{I}$. Assume that $dom (B_i)\subseteq dom (A_i)$, for all $i \in \mathbb{I}:=\{1,2,3, \cdots, m\}$ with $m\in \NN$. Choose any nonempty, closed and convex set, $X \subseteq \cap_{i\in \mathbb{I}}dom (B_i)$, satisfying  $X\cap S^*\ne \emptyset$. Thus, from now on, the solution set, $S^*$, is nonempty. Also we assume that the operators $B_i$ for all $i\in \mathbb{I}$ satisfies, that for each bounded subset $V$ of $dom(B_i)$ there exists $R>0$, such that $B_i(x)\cap B[0,R]\neq\emptyset$, for all $x\in V$ and $i\in \mathbb{I}$ where $B[0,R]$ is the closed ball centered in $0$ and radius $R$. We emphasize that this assumption holds trivially if $dom(B_i)=\RR^n$ or $V\subset int(dom(B_i))$ or $B_i$ is the normal cone in any subset of $dom(B_i)$.

Let $\{\beta_k\}_{k=0}^{\infty}$ be a sequence such that $\{\beta_k\}_{k\in \NN}\subseteq [\check{\beta},\hat{\beta}] $ with $0<\check{\beta} \leq \hat{\beta}<\infty$, $\theta, \delta\in(0,1)$, and be $\mathbb{I}=\{1,2,3, \cdots, m\}$, $R>0$ like assumption above. The algorithm is defined as follows:

\begin{center}
\fbox{\begin{minipage}[b]{\textwidth}
\begin{Calg}{A}\label{concep} Let $\{\beta_k\}_{k\in \NN}, \theta, \delta, R \mbox{ and } \mathbb{I}$ like above.
\begin{retraitsimple}
\item[] {\bf Step~0 (Initialization):} Take $x^0\in X$.

\item[] {\bf Step~1 (Iterative Step 1):} Given $x^k$, compute for all $i\in \mathbb{I}$,
\begin{equation}{\label{jota}}
J_i(x^k,\beta_{k}):=(I+\beta_{k}B_i)^{-1}(I-\beta_{k}A_i)(x^{k}).
\end{equation}
\item[]{\bf Step~2 (Stopping Test 1):} Define $\mathbb{I}_k^*:=\{i\in \mathbb{I}: x^k=J_{i}(x^k,\beta_k)\}$. If $\mathbb{I}_k^*=\mathbb{I}$ stop.

\item[] {\bf Step~3 (Inner Loop):} Otherwise, for all $i\in \mathbb{I}\setminus \mathbb{I}_k^*$ begin the inner loop over $j$.
 Put $j=0$ and choose any $u_{(j,i)}^{k}\in B_i\big(\theta^{j}J_i(x^{k},\beta_k)+(1-\theta^{j})x^k\big)\cap B[0,R]$. If
\begin{equation}\label{jk}
\Big \la A_i \big(\theta^{j}J_i(x^{k},\beta_k)+(1-\theta^{j})x^k\big)+u^{k}_{(j,i)}, x^k-J_i(x^k,\beta_k)\Big \ra\geq \frac{\delta}{\beta_k}\|x^k -J_i(x^k,\beta_k)\|^2,
\end{equation}
then $j_i(k):=j$ and stop.
Else, $j=j+1$.
\item[] {\bf Step~4 (Iterative Step 2):} Set for all $i\in \mathbb{I}\setminus \mathbb{I}_k^*$

\begin{equation}{\label{alphak}}
\alpha_{k,i}:=\theta^{j_i(k)},
\end{equation}
\begin{equation}{\label{ubar}}
\bar{u}_i^k:=u^k_{j_i(k)}
\end{equation}
\begin{equation}{\label{xbar}}
\bar{x}_i^k:=\alpha_{k,i} J_{i}(x^k,\beta_k)+(1-\alpha_{k,i})x^k
\end{equation}
and
\begin{equation}{\label{Fk}}
x^{k+1}:=\mathcal{F}_A(x^k).
\end{equation}
\item[] {\bf Step~5 (Stop Criteria 2):} If $x^{k+1}=x^k$ then stop. Otherwise, set $k\leftarrow k+1$ and go to {\bf Step~1}.
\end{retraitsimple}
\end{Calg}\end{minipage}}
\end{center}

\noindent We consider three variants of this algorithm. Their main difference lies in the computation \eqref{Fk}:
\begin{align}
\mcF_{\rm\ref{A1}}(x^k) =&P_X\big(P_{H_k}(x^k)\big) ;\label{P112}  \quad   &{(\bf Variant\; \ref{A1})} \\
\mcF_{\rm\ref{A2}}(x^k) =&P_{X\cap H_k}(x^k) ;\label{P122}  \quad   &{(\bf Variant\; \ref{A2})}\\
\mcF_{\rm\ref{A3}}(x^k) =&P_{X\cap H_k\cap W(x^k)}(x^0) ;\label{P132}  \quad   & {(\bf Variant\; \ref{A3})}
\end{align}
where
\begin{equation}\label{hk} H_k:=\cap_{i\in \mathbb{I}\setminus \mathbb{I}_k^*}H_i( \bar{x}_{i}^{k},\bar{u}_i^k)\end{equation}
\begin{equation}\label{H(x)}
H_i(x,u) := \big\{ y\in \RR^n :\la A_i(x)+u,y-x\ra\le 0\big \}
\end{equation}
and
\begin{equation}\label{W(x)}
W(x) := \big\{ y\in \RR^n : \la y-x,x^0-x\ra\le 0\big \}.
\end{equation}
This kind of hyperplane have been used in some works, see \cite{yuniusem,sva}.

\section{Convergence Analysis}\label{section4}
In this section we analyze the convergence of the algorithms presented in the previous section. First, we present some general properties as well as prove the well-definition of the conceptual algorithm.

\begin{lemma}\label{propseq}
For all $(x,u)\in Gr(B_i)$, $S_i^*\subseteq H_i(x,u)$, for all $i\in \mathbb{I}$. Therefore $S^* \subset H_i(x,u)$ for all $i \in \mathbb{I}$.
\end{lemma}

\begin{proof}
Take $x^{*}\in S_i^*$. Using the definition of the solution, there exists $v^{*}\in B_i(x^{*})$, such that $0=A_i(x^{*})+v^{*}$. By the monotonicity of $A_i+B_i$, we have
$$\la A_i(x)+u -(A_i(x^{*})+v^{*}), x-x^{*}\ra\ge 0, $$
 for all $(x,u)\in Gr(B_i)$.
Hence,
$$\la A_i(x)+u, x^{*}-x\ra \le 0$$
and by \eqref{H(x)}, $x^{*}\in H_i(x,u)$.
\end{proof}

\noindent From now on, $\{x^k\}_{k\in \NN}$ is the sequence generated by the conceptual algorithm.
\begin{proposition}\label{propdef}
 The conceptual algorithm  is well-defined.
\end{proposition}
\begin{proof}
By Proposition \ref{parada},  Stop Criteria $1$ is well-defined. The proof of the well-definition of $j_i(k)$ is by contradiction. Fix $i\in \mathbb{I}\setminus \mathbb{I}_k^*$ and assume that for all $j\ge0$ having chosen $u_{(j,i)}^{k}\in B_i\big(\theta^j J_i(x^k,\beta_k)+(1-\theta^j)x^k\big)\cap B[0,R]$,
\begin{equation*}
\Big\la A_i \big(\theta^{j}J_i(x^{k},\beta_k)+(1-\theta^{j})x^k\big)+u^{k}_{j}, x^k-J_i(x^k,\beta_k)\Big\ra < \frac{\delta}{\beta_k}\|x^k - J_i(x^k,\beta_k)\|^2.
\end{equation*}
Since the sequence $\{u^{k}_{(j,i)}\}_{j=0}^{\infty}$ is bounded, there exists a subsequence $\{u^{k}_{(\ell_j,i)}\}_{j=0}^{\infty}$ of $\{u^{k}_{(j,i)}\}_{j=0}^{\infty}$, which converges to an element $ u_i^k$ belonging to $B_i(x^k)$ by maximality. Taking the limit over the subsequence $\{\ell_j\}_{j\in \NN}$, we get
\begin{equation}{\label{lim}}
\big\la\beta_k A_i(x^k)+\beta_k u_i^k, x^k -J_i(x^k,\beta_k)\big \ra \le \delta \|x^k - J_i(x^k,\beta_k)\|^2.
\end{equation}
It follows from (\ref{jota}) that
\begin{equation*}{\label{res}}
 \beta_k A_i(x^k)=x^k-J_i(x^k,\beta_k)-\beta_k v_i^k,
 \end{equation*}
 for some  $ v_i^{k}\in B_i(J_i(x^k,\beta_k))$.\\
Now, the above equality together with (\ref{lim}), lead to
\begin{equation*}
\|x^k - J_i(x^k,\beta_k)\|^2\le\Big\la x^k-J_i(x^k,\beta_k)-\beta_k v_i^k+\beta_k u_i^k, x^k -J_i(x^k,\beta_k)\Big \ra \le \delta \|x^k - J_i(x^k,\beta_k)\|^2,
\end{equation*}
using the monotonicity of $B_i$ for the first inequality. So,
$$(1-\delta)\|x^k - J_i(x^k,\beta_k)\|^2\le 0,$$
which contradicts that $i\in \mathbb{I}\setminus \mathbb{I}_k^*$. Thus, the conceptual algorithm is well-defined.
\end{proof}

\begin{proposition}\label{H-separa-x} $x^k \in H_k$  if and only if, $x^k\in S^*$.
\end{proposition}
\begin{proof}
If $x^k\in H_k$ then $x^k\in H_i(\bar{x}_i^k,\bar{u}_i^k)$ for all $i \in \mathbb{I}\setminus \mathbb{I}_k^*$  by definition of $H_k$. Now by Proposition (4.2) of \cite{rei-yun} we have that $x^k \in S_i^*$ for all $i \in \mathbb{I}$, then $x^k\in S^*$. Conversely, if $x^k \in S^*$ then $x^k\in S_i^*$ then $x^k\in H_i(\bar{x}_i^k,\bar{u}_i^k)$ for all $i \in \mathbb{I}$ using the same proposition, implying that $x^k\in H_k$.
\end{proof}
Finally, a useful algebraic property on the sequence generated by the conceptual algorithm, which is a direct consequence of the inner loop and \eqref{xbar}.
\begin{corollary}\label{coro}
Let $\{x^k\}_{k\in \NN}$, $\{\beta_k\}_{k\in \NN}$ and $\{\alpha_{(k,i)}\}_{k\in \NN}$ be sequences generated by the conceptual algorithm. With $\delta$ and $\hat{\beta}$ as in the conceptual algorithm. Then,
\begin{equation}\label{desig-muy-usada}
\la A_i(\bar{x}_i^{k})+\bar{u}_i^{k},x^{k}-\bar{x}_i^{k} \ra  \ge\frac{\alpha_{k,i}\delta}{\hat{\beta}}\|x^{k}-J_i(x^{k},\beta_{k})\|^2\geq0,
\end{equation}
for all $k$.
\end{corollary}
\subsection{Convergence analysis of Variant \ref{A1}}\label{sec-5.1}
In this section, all results are for {\bf Variant \ref{A1}}, which is summarized below.

\vspace*{-0.2in}
\begin{center}\fbox{\begin{minipage}[b]{\textwidth}
\begin{variant}{A.1}
\label{A1}$x^{k+1}=\mcF_{\rm\ref{A1}}(x^k)=P_X\big(P_{H_k}(x^k)\big)$
\end{variant}\end{minipage}}\end{center}
\begin{proposition}\label{stop1}
If {\bf Variant \ref{A1}} stops, then $x^k\in S^*$.
\end{proposition}

\begin{proof}
If Stop Criteria $2$ is satisfied, $x^{k+1}=P_X\big(P_{H_k}(x^k)\big)=x^k$. Using Proposition \ref{proj}(ii), we have
\begin{equation}\label{proyex}
\la P_{H_k}(x^k)-x^k, z-x^k\ra \leq 0,
\end{equation} for all $z\in X$. Now using Proposition \ref{proj}(ii),
\begin{equation}\label{proyeh}
\la P_{H_k}(x^k)-x^k, P_{H_k}(x^k)-z\ra \leq 0,
\end{equation} for all $z\in H_k$.
Since $X\cap H_k \neq \emptyset$ summing \eqref{proyex} and \eqref{proyeh}, with $z\in X\cap H_k$, we get
\begin{equation*}
\|x^k-P_{H_k}(x^k)\|^2=0.
\end{equation*}
Hence, $x^k=P_{H_k}(x^k)$, implying that $x^k\in H_k$ and by Proposition \ref{H-separa-x}, $x^k\in S^*$.
\end{proof}

\begin{proposition}\label{prop2}
\begin{enumerate}
\item The sequence $\{x^k\}_{k\in \NN}$ is Fej\'er convergente to $S^*\cap X$.
\item The sequence $\{x^k\}_{k\in \NN}$ is bounded.
\item $\lim_{k\to \infty}\|P_{H_k}(x^k)-x^k\|^2=0$.
\item $\lim_{k\to \infty}\|x^{x+1}-x^k\|^2=0$.
\end{enumerate}
\end{proposition}
\begin{proof}
(i) Take $x^*\in S^*\cap X$.  Using \eqref{P112}, Proposition \ref{proj}(i) and Lemma \ref{propseq}, we have
\begin{eqnarray}\label{fejer-des}\nonumber\|x^{k+1}-x^{*}\|^2&=&\|P_{X}(P_{H_k}(x^k))-P_{X}(P_{H_k}(x^{*}))\|^2\le \|P_{H_k}(x^k)-P_{H_k}(x^{*})\|^2\\&\leq& \|x^k-x^*\|^2-\|P_{H_k}(x^k)-x^k\|^2.\end{eqnarray} So, $\|x^{k+1}-x^{*}\|\le \|x^k-x^*\|$.
(ii) Follows immediately from item (i).
(iii)Take $x^* \in S^*\cap X$. Using \eqref{fejer-des} yields
\begin{equation}\label{ineq}
\|P_{H_k}(x^k)-x^k\|^2\le \|x^k-x^*\|^2-\|x^{k+1}-x^{*}\|^2.
\end{equation}
Now using Proposition \ref{punto} and item (i) we have that the right side of equation \eqref{ineq} go to zero. Obtaining the result.
(iv) Since the sequence $\{x^k\}_{k\in \NN}$ belong to $X$, we have,
$$\|x^{k+1}-x^k\|^2=\|P_{X}(P_{H_k}(x^k))-P_{X}(x^k)\|^2\le \|P_{H_k}(x^k)-x^k\|^2.$$
Taking limits in the above equation and using the previous item we have the result.
\end{proof}
\begin{proposition}\label{cadai}
For all $i\in \mathbb{I}$ we have,
$$\lim_{k\to \infty}\la A_i(\bar{x}_i^k)+\bar{u}_i^k,x^k-\bar{x}_i^k \ra=0.$$
\end{proposition}
\begin{proof}
For all $i\in \mathbb{I}$. Using Proposition \ref{proj}(i) and the fact that $H_k \subset H(\bar{x}_i^k,\bar{u}_i^k)$ by \eqref{hk}, we have that,
\begin{align}\label{todoi}
\|x^{k+1}-x^*\|^2 =&\|P_{X}(P_{H_k}(x^k))-P_{X}(x^*)\|^2\le \|P_{H_k}(x^k)-x^*\|^2 \nonumber \\
=&\|P_{H_k}(x^k)-P_{H(\bar{x}_i^k,\bar{u}_i^k)}(x^k)+P_{H(\bar{x}_i^k,\bar{u}_i^k)}(x^k)-x^*\|^2\nonumber \\
\le& \|P_{H_k}(x^k)-x^k\|^2+\|P_{H(\bar{x}_i^k,\bar{u}_i^k)}(x^k)-x^*\|^2.
\end{align}
Now using Proposition \ref{proj}(i) and reordering \eqref{todoi}, we get,
$$\|P_{H(\bar{x}_i^k,\bar{u}_i^k)}(x^k)-x^k\|^2\leq \|x^k-x^*\|^2-\|x^{k+1}-x^*\|^2+\|P_{H_k}(x^k)-x^k\|^2.$$
Using the fact that,
$$P_{H(\bar{x}_i^k,\bar{u}_i^k)}(x^k)=x^k-\frac{\la A_i(\bar{x}_i^k)+\bar{u}_i^k,x^k-\bar{x}_i^k  \ra}{\|A_i(\bar{x}_i^k)+\bar{u}_i^k\|^2}(A_i(\bar{x}_i^k)+\bar{u}_i^k),$$
and the previous equation, we have,
\begin{equation}\label{pasar al lim}
\frac{\big(\la A_i(\bar{x}_i^k)+\bar{u}_i^k,x^k-\bar{x}_i^k  \ra\big)^2}{\|A_i(\bar{x}_i^k)+\bar{u}_i^k\|^2}\leq \|x^k-x^*\|^2-\|x^{k+1}-x^*\|^2+\|P_{H_k}(x^k)-x^k\|^2.
\end{equation}
By Proposition \ref{inversa} and the continuity of $A_i$ we have that $J_i$ is continuo, since $\{x^k\}_{k\in \NN}$ and $\{\beta_k\}_{k\in \NN}$ are bounded then $\{J_i(x^k,\beta_k)\}_{k\in \NN}$ and $\{\bar{x}_i^k\}_{k\in \NN}$ are bounded, implying the boundedness of $\{\|A_i(\bar{x}_i^k)+\bar{u}_i^k\|\}_{k\in \NN}$ for all $i\in \mathbb{I}$.

\noindent Using Proposition \ref{punto}(ii) and (iii), the right side of \eqref{pasar al lim} goes to 0, when $k$ goes to $\infty$, establishing the result.
\end{proof}

\noindent Next we establish our main convergence result on {\bf Variant \ref{A1}}.
\begin{theorem}\label{teo1}
The sequence $\{x^k\}_{k\in \NN}$ converges to some element belonging to $S^*\cap X $.
 \end{theorem}
\begin{proof}
 We claim that there exists a cluster point of $\{x^k\}_{k\in \NN}$ belonging to $S^*$. The existence of the cluster points follows from Proposition \ref{prop2}(ii). Let $\{x^{j_k}\}_{k\in \NN}$ be a convergent subsequence of $\{x^k\}_{k\in \NN}$ such that, for all $i\in \mathbb{I}$ the sequences  $\{\bar{x}_i^{j_k}\}_{k\in \NN}, \{\bar{u}_i^{j_k}\}_{k\in \NN}, \{\alpha_{j_k,i}\}_{k\in \NN}$ and $\{\beta_{j_k}\}_{k\in \NN}$ are convergents, and $\lim _{k\to \infty}x^{j_k}= \tilde{x}$.\\
Using Proposition \ref{prop2}(iii) and taking limits in \eqref{desig-muy-usada} over the subsequence $\{j_k\}_{k\in \NN}$, we have for all $i\in \mathbb{I}$,
\begin{equation}\label{limite}
0=\lim_{k\to \infty}\la A_i(\bar{x}_i^{j_k})+\bar{u}_i^{j_k},x^{j_k}-\bar{x}_i^{j_k} \ra \ge \lim_{k\to \infty}  \frac{\alpha_{j_k,i}\delta}{\hat{\beta}}\|x^{j_k}-J_i(x^{j_k},\beta_{j_k})\|^2\geq0.
\end{equation}
Therefore,
\begin{equation*}
\lim_{k\to \infty} \alpha_{j_k,i}\|x^{j_k}-J_i(x^{j_k},\beta_{j_k})\|=0.
\end{equation*}
Now consider the two possible cases.

(a) First, assume that $\lim_{k\to \infty}\alpha_{j_k,i}\ne 0$, i.e., $\alpha_{j_k,i}\ge \bar{\alpha}$ for all $k$ and some  $\bar{\alpha}>0$. In view of (\ref{limite}),

    \begin{equation}\label{limcero}
    \lim_{k\to \infty}\|x^{j_k}-J_i(x^{j_k},\beta_{j_k})\|=0.
    \end{equation}
    Since $J_i$ is continuous, by the continuity of $A_i$ and $(I+\beta_k B_i)^{-1}$ and by Proposition \ref{inversa}, \eqref{limcero} becomes
    \begin{equation*}
    \tilde{x}=J_i(\tilde{x},\tilde{\beta}),
    \end{equation*}
    which implies that $\tilde{x}\in S_i^*$ for all $i\in \mathbb{I}$. Then $\tilde{x}\in S^*$ establishing the claim.

 (b) On the other hand, if $\lim_{k\to \infty}\alpha_{j_k,i}=0$ then for $\theta \in (0,1)$ as in the conceptual algorithm, we have
    $$\lim_{k\to\infty}\frac{\alpha_{j_k,i}}{\theta}=0.$$
    Define
    $$y^{j_k}_i:=\frac{\alpha_{j_k,i}}{\theta}J_i(x^{j_k},\beta_{j_k})+\Big(1-\frac{\alpha_{j_k,i}}{\theta}\Big)x^{j_k}.$$
    Then,
    \begin{equation}\label{ykgox}
    \lim_{k\to\infty}y_i^{j_k}=\tilde{x}.
    \end{equation}
    Using the definition of the $j_i(k)$ and \eqref{alphak}, we have that $y_i^{j_k}$ does not satisfy \eqref{jk} implying
    \begin{equation*}
    \Big \la A_i (y^{j_k}_i)+u^{j_k}_{j_i(k)-1}-\frac{\delta}{\beta_k}(x^k -J_i\big(x^k,\beta_k)\big), x^k-J_i(x^k,\beta_k)\Big \ra > 0,
    \end{equation*}
    equivalent to
    \begin{equation}\label{conse}
\Big \la A_i (y^{j_k}_i)+u^{j_k}_{j(j_k)-1,i}, x^k-J_i(x^k,\beta_k)\Big \ra > \frac{\delta}{\beta_k}\|x^k-J_i\big(x^k,\beta_k\big)\|^2,
    \end{equation}
    for $u^{j_k}_{j(j_k)-1,i}\in B_i(y^{j_k}_i)$ and all $k\in \NN$ and $i\in \mathbb{I}$.\\
    Redefining the subsequence $\{j_k\}_{k\in \NN}$, if necessary, we may assume that  $\{u^{j_k}_{j(j_k)-1,i}\}_{k\in \NN}$ converges to $\tilde{u}_i$. By the maximality of $B_i$, $\tilde{u}_i$ belongs to $B_i(\tilde{x})$. Using the continuity of $J_i$, $\{J(x^{j_k},\beta_{j_k})\}_{k\in \NN}$ converges to $J_i(\tilde{x},\tilde{\beta})$. Using \eqref{ykgox} and taking limit in (\ref{conse}) over the subsequence $\{j_k\}_{k\in \NN}$, we have
    \begin{equation}\label{tiu}
    \Big\la A_i(\tilde{x}) + \tilde{u}_i, \tilde{x}- J_i(\tilde{x},\tilde{\beta})\Big\ra \le \frac{\delta}{\tilde{\beta}}\|\tilde{x}-J_i(\tilde{x},\tilde{\beta})\|^2.
    \end{equation}
    Using (\ref{jota}) and multiplying by $\tilde{\beta}$ on both sides of (\ref{tiu}), we get
    \begin{equation*}
    \la \tilde{x}-J_i(\tilde{x},\tilde{\beta})-\tilde{\beta}\tilde{v}_i+\tilde{\beta}\tilde{u}_i, \tilde{x}-J_i(\tilde{x},\tilde{\beta})\ra \le \delta\| \tilde{x}-J_i(\tilde{x},\tilde{\beta})\|^2,
    \end{equation*}
    where $\tilde{v}_i\in B_i(J_i(\tilde{x},\tilde\beta))$. Applying the monotonicity of $B_i$, we obtain
    $$\| \tilde{x}-J_i(\tilde{x},\tilde{\beta})\|^2 \le \delta \| \tilde{x}-J_i(\tilde{x},\tilde{\beta})\|^2,$$
    implying that $\| \tilde{x}-J_i(\tilde{x},\tilde{\beta})\|\le 0$. Thus, $\tilde{x}=J_i(\tilde{x},\tilde{\beta})$ and hence, $\tilde{x}\in S_i^*$ for all $i\in \mathbb{I}$, thus $\tilde{x} \in S^*$.
\end{proof}
\subsection{Convergence analysis of Variant \ref{A2}}\label{sec-5.2}
In this section, all results are for {\bf Variant
\ref{A2}}, which is summarized below.

\vspace*{-0.2in}
\begin{center}\fbox{\begin{minipage}[b]{\textwidth}
\begin{variant}{A.2}
\label{A2}$x^{k+1}=\mcF_{\rm\ref{A2}}(x^k)=P_{X\cap H_k}(x^k)$
\end{variant}\end{minipage}}\end{center}
\begin{proposition}\label{stop2}
If {\bf Variant \ref{A2}} stops, then $x^k\in S^*$.
\end{proposition}
\begin{proof}
If $x^{k+1}=P_{X\cap H_k}(x^k)=x^k$ then $x^k\in X\cap H_k$ and by Proposition \ref{H-separa-x}, $x^k\in S^*\cap X$.
\end{proof}
\noindent From now on assume that {\bf Variant \ref{A2}} does not stop.
\begin{proposition}\label{fe}
The sequence $\{x^k\}_{k\in \NN}$ is F\'ejer convergent to $S^*\cap X$. Moreover, it is bounded and
\begin{equation*}
\lim_{k \to \infty} \|x^{k+1}-x^k\|=0.
\end{equation*}
\end{proposition}
\begin{proof}
Take $x^*\in S^*\cap X$. By Lemma \ref{propseq}, $x^*\in H_k\cap X$, for all $k$. Then using Proposition \ref{proj}(ii) and \eqref{P122}
\begin{equation*}
\|x^{k+1}-x^*\|^2-\|x^k-x^*\|^2+\|x^{k+1}-x^k\|^2=2\la x^*-x^{k+1},x^k-x^{k+1}\ra\leq 0,
\end{equation*}
we obtain
\begin{equation}\label{fejerc}
\|x^{k+1}-x^*\|^2\le \|x^k-x^*\|^2-\|x^{k+1}-x^k\|^2.
\end{equation}
The above inequality implies that $\{x^k\}_{k\in \NN}$ is F\'ejer convergent to $S^*\cap X$. Hence by Proposition \ref{punto}(i) and (ii), $\{x^k\}_{k\in \NN}$ is bounded and thus $\{\|x^k-x^*\|\}_{k\in \NN}$ is a convergent sequence. Taking limits in (\ref{fejerc}), we get
\begin{equation*}
\lim_{k \to \infty} \|x^{k+1}-x^k\|=0.
\end{equation*}
\end{proof}

The next proposition shows a relation between the projection steps in {\bf Variant \ref{A1}} and {\bf \ref{A2}}. This fact has a geometry interpretation, since the projection of {\bf Variant \ref{A2}} is done over a small set, improving the convergence of {\bf Variant \ref{A1}}. Note that this can be reduce the number of iterations, avoiding possible zigzagging of {\bf Variant \ref{A1}}.
\begin{proposition}Let $\{x^{k}\}_{k\in \NN}$ the sequence generated by {\bf Variant \ref{A2}}. Then,
\begin{enumerate}
\item $x^{k+1}=P_{X\cap H_k}(P_{H_k}(x^k))$.
\item For all $i\in \mathbb{I}$ we have,  $\lim_{k\to \infty}\la A_i(\bar{x}_i^k)+\bar{u}_i^k,x^k-\bar{x}_i^k \ra=0$.
\end{enumerate}
\end{proposition}

\begin{proof}
\noindent (i) Fix any $y\in X\cap H_k$. Since $x^k \in X$ but $x^k\notin H_k$ by Proposition \ref{H-separa-x}, there exists $\gamma \in [0,1]$, such that $\tilde{x}=\gamma x^k+(1-\gamma)y\in X\cap \partial H_k$. Hence,
\begin{eqnarray}
\|y-P_{H_k}(x^k)\|^2&\geq & (1-\gamma)^2\|y-P_{H_k}(x^k)\|^2\nonumber\\
&=& \|\tilde{x}-\gamma x^k-(1-\gamma) P_{H_k}(x^k)\|^2\nonumber\\
&=& \|\tilde{x}-P_{H_k}(x^k)\|^2+\gamma^2\|x^k-P_{H_k}(x^k)\|^2 -2\gamma\la\tilde{x}-P_{H_k}(x^k),x^k-P_{H_k}(x^k)\ra\nonumber\\
&\geq& \|\tilde{x}-P_{H_k}(x^k)\|^2,\label{zetabar}
\end{eqnarray}
where the last inequality follows from Proposition \ref{proj}(ii), applied with $X=H_k$, $x=x^k$ and $z=\tilde{x}\in H_k$. Furthermore, we have
\begin{eqnarray}
\|\tilde{x}-P_{H_k}(x^k)\|&\geq&\|\tilde{x}-x^k\|-\|x^k-P_{H_k}(x^k)\|\nonumber\\
&\geq& \|x^{k+1}-x^k\|-\|x^k-P_{H_k}(x^k)\|\nonumber\\
&\geq&  \|x^{k+1}-x^k\|\nonumber\\
&\geq& \|x^{k+1}-P_{H_k}(x^k)\|\label{proye},
\end{eqnarray}
where the first equality follows by the triangle inequality, using the fact that $\tilde{x}\in X\cap H_k$ and $x^{k+1}=P_{X\cap H_k}(x^k)$ in the second inequality, the third one is trivial, and the last one inequality by the fact that $x^{k+1}\in H_k$ and Proposition \ref{proj}(i) with $X=H_k$. Combining (\ref{zetabar}) and (\ref{proye}), we obtain
\begin{equation*}
\|y-P_{H_k}(x^k)\|\geq \|x^{k+1}-P_{H_k}(x^k)\|,
\end{equation*}
for all $y\in X\cap H_k$.
Hence, $x^{k+1}=P_{X\cap H_k}(P_{H_k}(x^k))$.

\medskip

\noindent (ii) Take $x^*\in X\cap S^*$. By item (i), Lemma \ref{propseq} and Proposition \ref{proj}(i), we have
\begin{equation*}
\|x^{k+1}-x^*\|^2=\|P_{X\cap H_k}(P_{H_k}(x^k))-P_{X\cap H_k}(x^*)\|^2\leq \|P_{H_k}(x^k)-x^*\|^2.
\end{equation*}

\noindent The proof is similar to the proof of Proposition \ref{cadai}.
\end{proof}

\noindent Finally we present the convergence result for {\bf Variant \ref{A2}}.
\begin{theorem}\label{teo2}
The sequence $\{x^k\}_{k\in \NN}$ converges to some point belonging to $S^*\cap X$.
\end{theorem}
\begin{proof}
Repeat the proof of Theorem \ref{teo1}.
\end{proof}

\subsection{Convergence analysis of Variant \ref{A3}}\label{sec-5.3}
In this section, all results are for {\bf Variant \ref{A3}}, which is summarized below.

\vspace*{-0.2in}
\begin{center}\fbox{\begin{minipage}[b]{\textwidth}
\begin{variant}{A.3}
\label{A3}$x^{k+1}=\mcF_{\rm\ref{A3}}(x^k)=P_{X\cap H_k\cap W(x^k)}(x^0)$
\end{variant}\end{minipage}}\end{center}
\begin{proposition}
If {\bf Variant \ref{A3}} stops, then $x^k\in S^*\cap X$.
\end{proposition}

\begin{proof}
If Stop Criteria 2 is satisfied then, $x^{k+1}=P_{X\cap H_k\cap W_k}(x^0)=x^k$. So, $x^k\in X\cap H_k\cap W_k\subset X\cap H_k $ and finally using Proposition \ref{H-separa-x}, $x^k\in S^*\cap X$.
\end{proof}

From now on we assume that {\bf Variant \ref{A3}} does not stop. Observe that, in virtue of their definitions, $W_k$ and $H_k$ are convex and closed sets, for each $k$. Therefore $X\cap H_k\cap W_k$ is a convex and closed set. So, if $X\cap H_k\cap W_k$ is nonempty, then the next iterate, $x^{k+1}$, is well-defined.
The following lemma guarantees this fact.

\begin{lemma}\label{lemma:3} $S^*\cap X\subset  H_k\cap W_k$, for all $k$.
\end{lemma}
\begin{proof} We proceed by induction. By definition, $S^*\cap X \neq \emptyset $.   By Lemma \ref{propseq}, $S^*\cap X\subset  H_k$, for all $k$.
For $k=0$, as $W_0=\RR^n$, $S^*\cap X\subset H_0\cap W_0$.

Assume that $S^*\cap X\subset H_\ell \cap W_\ell$, for $\ell\leq k$.
Henceforth, $x^{k+1}=P_{X\cap H_k\cap W_k}(x^0)$ is well-defined.
Then, by Proposition \ref{proj}(ii), we have
\begin{equation}\label{x*ink+1}
\langle x^*-x^{k+1}\,,\, x^0-x^{k+1}\rangle=\langle x^*-P_{X\cap
H_k\cap W_k}(x^0)\,,\, x^0-P_{X\cap H_k\cap W_k}(x^0)\rangle\leq0,
\end{equation}
for all $x^*\in S^*\cap X$. The inequality follows by the induction hypothesis. Now, \eqref{x*ink+1} implies that
$x^*\in W_{k+1}$ and hence, $ S^*\cap X \subset H_{k+1}\cap W_{k+1}$.
\end{proof}

The above lemma shows that the set $X\cap H_k\cap W_k$ is  nonempty and in consequence the projection step, given in  \eqref{P132}, is well-defined.
\begin{corollary}\label{l:well-definedness}  {\bf Variant \ref{A3}} is well-defined.
\end{corollary}
\begin{proof} By Lemma \ref{lemma:3} , $ S^*\cap X\subset H_k\cap W_k$, for  all $k$. Then,  given $x^0$, the sequence $\{x^k\}_{k\in \NN}$ is computable.
\end{proof}

Before proving the convergence of the sequence, we study its
boundedness. The next lemma  shows that  the sequence remains in a
ball determined by the initial point.

\begin{lemma}\label{l:limitacao} The sequence $\{x^k\}_{k\in \NN}$ is bounded. Furthermore,
\begin{equation*}\label{eq:bolas}
\{x^k\}_{k\in \NN}\subset  B\left[\frac{1}{2}(x^0+\bar{x}),\frac{1}{2}\rho\right]\cap X,
\end{equation*}
where $\bar{x}=P_{S^*\cap X}(x^0)$ and $\rho={\rm dist}(x^0,
S^*\cap X)$.
\end{lemma}
\begin{proof}
$S^*\cap X \subset H_k \cap W_k$ follows from Lemma \ref{lemma:3}. Moreover, from \eqref{P132}, we obtain that
\begin{equation}\label{eq:12}
\| x^{k+1}-x^0\| \leq\| z-x^0\|,
\end{equation}
for all $k$ and all $z\in S^*\cap X$. Henceforth, taking $z=\bar{x}$ in \eqref{eq:12},
\begin{equation}\label{eq:yunier}
 \| x^{k+1}-x^0\|\leq\|\bar{x}-x^0\|=\rho,
\end{equation}
for all $k$. Thus, $\{x^k\}_{k\in \NN}$ is bounded.
Define $z^{k}=x^{k}-\frac{1}{2}(x^0+\bar{x})$ and $\bar{z}=\bar{x}-\frac{1}{2}(x^0+\bar{x})$. It follows from the fact $\bar{x}\in W_{k+1}$, that
\begin{eqnarray*}
0&\geq& 2\la \bar{x}-x^{k+1},x^0-x^{k+1}\ra \\&=&2\left\la
\bar{z}+ \frac{1}{2}(x^0+\bar{x})-z^{k+1}-\frac{1}{2}(x^0+\bar{x}),z^0+\frac{1}{2}(x^0+\bar{x})-z^{k+1}-\frac{1}{2}(x^0+\bar{x})\right\ra\\&=&2\left\la
\bar{z}-z^{k+1},z^0-z^{k+1}\right\ra=\left\la
\bar{z}-z^{k+1},-\bar{z}-z^{k+1}\right\ra = \|z^{k+1}\|^2-\|\bar{z}\|^2,
\end{eqnarray*}
where we have used that $\bar{z}=-z^0$ in the third equality. So,
\begin{equation*}\label{eq:raio}
\left \|x^{k+1}-\frac{x^0+\bar{x}}{2}\right\|\leq\left\|\bar{x}-\frac{x^0+\bar{x}}{2}\right\|=\frac{\rho}{2},
\end{equation*}
for all $k$. Now, the result follows from the feasibility of $\{x^k\}_{k\in \NN}$, which, in turn, is a consequence of \eqref{P132}.
\end{proof}

\noindent  Now, we focus on the properties of the accumulation points.

\begin{lemma}\label{l:optimalidad} All accumulation points of $\{x^k\}_{k\in \NN}$ belong to $S^*\cap X$.
\end{lemma}

\begin{proof}
Since $x^{k+1}\in W_k$,
\begin{equation*}
0\geq 2 \la
x^{k+1}-x^k,x^0-x^k\ra=\|x^{k+1}-x^k\|^2-\|x^{k+1}-x^0\|^2+\|x^k-x^0\|^2.
\end{equation*}
Equivalently $$0\leq\|x^{k+1}-x^k\|^2\leq\|x^{k+1}-x^0\|^2-\|x^k-x^0\|^2,$$ establishing that the sequence $\{\|x^k-x^0\|\}_{k\in \NN}$ is monotone and nondecreasing.
From Lemma \ref{l:limitacao}, we get that $\{\|x^k-x^0\|\}_{k\in \NN}$ is bounded, and thus, convergent. Therefore,
\begin{equation}\label{xk+1-xk-va-cero}
\lim_{k\rightarrow\infty}\| x^{k+1}-x^k\|=0.
\end{equation}
Since $x^{k+1}\in H_k$, we get for all $i\in \mathbb{I}$ that,
\begin{equation}\label{ptos-in-s*}
 \la A_i(\bar{x}_i^{k})+\bar{u}_i^{k},x^{k+1}-\bar{x}_i^{k}\ra\le 0,
\end{equation}
with $\bar{u}_i^k$ and $\bar{x}_i^k$ as \eqref{ubar} and \eqref{xbar}.

\noindent Using \eqref{xbar} and \eqref{ptos-in-s*}, we have
\begin{equation*}
\la A_i(\bar{x}_i^{k})+\bar{u}_i^{k},x^{k+1}-x^{k}\ra + \alpha_{k,i}\big{\la} A_i(\bar{x}_i^{k})+\bar{u}_i^{k},x^{k}-J_i(x^{k},\beta_{k})\big{\ra}\le 0.
\end{equation*}
Combining the above inequality with the stop criteria of Inner Loop, given in \eqref{jk}, we get for all $i\in \mathbb{I}$
\begin{equation}\label{eq}
\la A_i(\bar{x}_i^{k})+\bar{u}_i^{k},x^{k+1}-x^{k}\ra+\frac{\alpha_{k,i}\delta}{\hat{\beta}}\|x^{k} - J_i(x^{k},\beta_{k})\|^2\leq 0.
\end{equation}
Choosing a subsequence $\{j_k\}_{k\in \NN}$ such that the subsequences $\{x^{j_k}\}_{k\in \NN}$, $\{\beta_{j_k}\}_{k\in \NN}$ and $\{\bar{u}_i^{j_k}\}_{k\in \NN}$ converge to  $\tilde{x}$, $\tilde{\beta}$ and $\tilde{u}_i$ respectively. This is possible by the boundedness of $\{\bar{u}_i^k\}_{k\in \NN}$, by hypothesis on $B_i$, bounded of $\{x^{k}\}_{k\in \NN}$ and $\{\beta_{k}\}_{k\in \NN}$. Taking limits in \eqref{eq}, we have
 \begin{equation}\label{zero}
\lim_{k\to \infty}\alpha_{j_k,i}\|x^{j_k} - J_i(x^{j_k},\beta_{j_k})\|^2=0.
\end{equation}
Now we consider two cases, $\lim_{k\to \infty} \alpha_{j_k,i}=0$ or $\lim_{k\to \infty} \alpha_{j_k,i}\neq 0$ (taking a subsequence again if necessary).

(a) $\lim_{k\to\infty} \alpha_{j_k,i} \neq 0$, i.e., for all $i\in \mathbb{I}$, $\alpha_{j_k,i}\geq \tilde{\alpha_i}$ for all $k$ and some $\tilde{\alpha_i}>0$. By \eqref{zero},
\begin{equation*}
\lim_{k\to \infty}\|x^{i_k} - J(x^{i_k},\beta_{i_k})\|^2=0.
\end{equation*}
By continuity of $J_i$, we have $\tilde{x}=J_i(\tilde{x},\tilde{\beta})$ and hence by Proposition \ref{parada}, $\tilde{x}\in S_i^*$ for all $i\in \mathbb{I}$, therefor $\tilde{x}\in S_*$.

(b)  $\lim_{k\to\infty} \alpha_{j_k,i} = 0$, then $\lim_{k\to \infty}\frac{\alpha_{j_k,i}}{\theta}=0$. It follows in the same, manner as in the proof of Theorem \ref{teo1}(b).
\end{proof}

\noindent Finally, we are ready to prove the convergence of the sequence $\{x^k\}_{k\in \NN}$ generated by {\bf Variant \ref{A3}}, to the solution closest to
$x^0$.

\begin{theorem}  Define $\bar{x}=P_{S^*\cap X}(x^0)$. Then, $\{x^k\}_{k\in \NN}$ converges to $\bar{x}$.
\end{theorem}
\begin{proof}   By  Lemma \ref{l:limitacao}, $\{x^k\}_{k\in \NN}\subset  B\left[\frac{1}{2}(x^0+\bar{x}),\frac{1}{2}\rho\right]\cap X$, so it is bounded. Let
$\{x^{j_k}\}_{k\in \NN}$ be a convergent subsequence of $\{x^k\}_{k\in \NN}$,
and let $\hat{x}$ be its  limit. Evidently $\hat{x}\in  B\left[\frac{1}{2}(x^0+\bar{x}),\frac{1}{2}\rho\right]\cap X$. Furthermore, by Lemma \ref{l:optimalidad}, $\hat{x}\in S^*\cap X$.
Then, $$\hat{x}\in S^* \cap X \cap  B \left[\frac{1}{2}(x^0+\bar{x}),\frac{1}{2}\rho\right]=\{\bar{x}\},$$ implying that $\hat{x}=\bar{x}$, hence $\bar{x}$ is the
unique limit point of $\{x^k\}_{k\in \NN}$. Thus,
$\{x^{k}\}_{k\in \NN}$ converges   to $\bar{x}\in S^*\cap X$. \end{proof}

\section{Conclusions}
In this paper, we present a variant of forward-backward splitting methods for solving a system o inclusion problems composed by the sum of two operators. A conceptual algorithm have been proposed containing three variants with different projections steps. A linesearch, for relax the hypothesis of Lipschitz continuity on forwards operators, have been proposed. The convergence analyse of three variant are discussed. The results presented here, improve the previous in the literature by relaxing the hypothesis.


\bibliographystyle{plain}

\end{document}